\documentclass[
final
]{dmtcs-episciences}

\usepackage[utf8]{inputenc}
\usepackage[american]{babel}
\usepackage{amsmath}
\usepackage{mathtools}
\usepackage[export]{adjustbox}
\usepackage{tikz}

\mathcode`@="8000
{\catcode`\@=\active\gdef@{\mkern1mu}}

\DeclareMathOperator{\tw}{tw}
\DeclareMathOperator{\cro}{cro}

\makeatletter
\newlength{\negph@wd}
\DeclareRobustCommand{\negphantom}[1]{%
  \ifmmode
    \mathpalette\negph@math{#1}%
  \else
    \negph@do{#1}%
  \fi
}
\newcommand{\negph@math}[2]{\negph@do{$\m@th#1#2$}}
\newcommand{\negph@do}[1]{%
  \settowidth{\negph@wd}{#1}%
  \hspace*{-\negph@wd}%
}
\makeatother

\newtheorem{theorem}{Theorem}
\newtheorem{definition}[theorem]{Definition}
\newtheorem{observation}[theorem]{Observation}
\newtheorem{lemma}[theorem]{Lemma}
\newtheorem{corollary}[theorem]{Corollary}
\newtheorem{example}[theorem]{Example}
\definecolor{viridis7blue}{RGB}{61, 97, 136}
\hypersetup{colorlinks=true, linkcolor=viridis7blue!90!black, citecolor=viridis7blue!90!black, urlcolor=black, filecolor=black}

\usepackage[round]{natbib}
\setcitestyle{numbers}

\author{Frank Göring\affiliationmark{1}
  \and Tobias Hofmann\affiliationmark{2}\thanks{My research was partly funded by the Deutsche Forschungsgemeinschaft (DFG, German Research Foundation) -- Project-ID 416228727 -- SFB 1410.}}
\title[Properties of uniformly $3$-connected graphs]{Properties of uniformly $3$-connected graphs}
\affiliation{
  Chemnitz University of Technology, Chemnitz, Germany\\
  TU Berlin, Berlin, Germany}

\begin{document}
\publicationdata
{vol. 25:3 special issue ICGT’22}
{2023}
{3}
{10.46298/dmtcs.10472}
{2023-05-05; None}
{2023-10-19}

\maketitle

\begin{abstract}
  A graph on at least~${{k+1}}$ vertices is uniformly $k$-connected if each pair of its vertices is connected by $k$ and not more than $k$ independent paths. We reinvestigate a recent constructive characterization of uniformly $3$-connected graphs and obtain a more detailed result that relates the number of vertices to the operations involved in constructing a respective uniformly \mbox{$3$-connected} graph. Furthermore, we investigate how crossing numbers and treewidths behave under the mentioned constructions. We demonstrate how these results can be utilized to study the structure and properties of uniformly $3$-connected graphs with minimum number of vertices of minimum degree.
\end{abstract}

\section{Introduction}
Among the many connectivity concepts in graph theory, requiring the same connectivity between each pair of a graph's vertices may seem to be quite restrictive. Yet it might be a valuable feature of certain communication or supply networks and, from a theoretical point of view, uniform connectivity nicely complements the notions of ordinary, minimal, or average connectivity. When studying the latter, Beineke, Oellermann, and Pippert~\cite{beineke2002average} introduced uniformly connected graphs as they became interested in determining for which graphs the connectivity equals the average connectivity. Let us begin by recalling the following definition, whereas we refer to Diestel~\cite{diestel2017graph} for basic graph theoretical terminology.

\begin{definition}
For a number~$k\in\mathbb{N}$ a graph on at least~${{k+1}}$ vertices is called \emph{uniformly $k$-connected} if each pair of its vertices is connected by $k$ and not more than $k$ independent paths.
\end{definition}

It is not hard to see that uniformly $1$-connected graphs are exactly all trees and uniformly $2$-connected graphs are exactly all cycles. Further examples are wheel graphs for~${{k=3}}$ or $k$-regular, \mbox{$k$-connected} graphs for~${{k\in\mathbb{N}}}$. In more detail, such relations as well as uniformly edge-connected graphs, in which each pair of vertices is connected by $k$ and not more than $k$ edge-disjoint paths, are discussed by Göring, Hofmann, and Streicher~\cite{hofmann2022uniformly}. This article also contains the following characterization.\newpage

\begin{theorem}\label{thm:constr}
A graph is uniformly $3$-connected if and only if it is contained in the following recursively defined class~$\mathcal{C}$. 
\begin{figure}[t]
    \centering
    \includegraphics{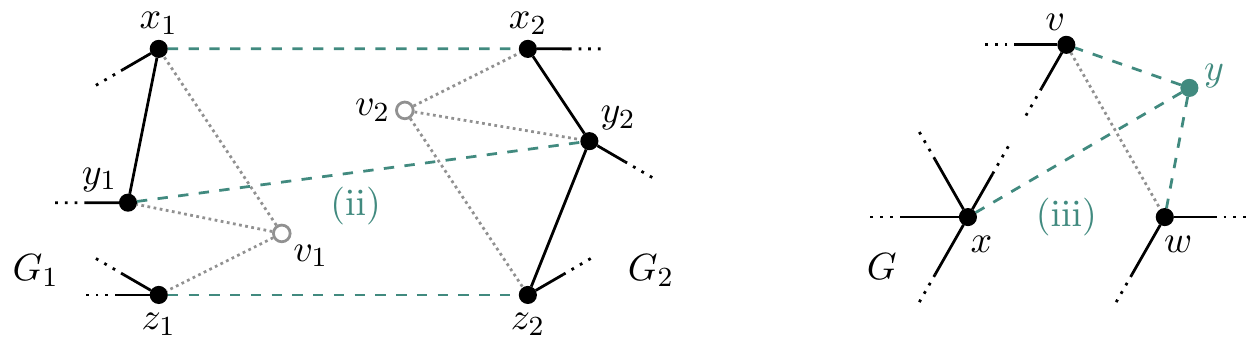}
    \caption{Constructing uniformly $3$-connected graphs}
    \label{fig:operations}
\end{figure}
\begin{enumerate}[(i)]
    \item If a graph $G$ is $3$-regular and $3$-connected, then $G$ shall be contained in $\mathcal{C}$.\label{thm:constr:i}
    \item For two graphs ${{G_1, G_2\in\mathcal{C}}}$ with vertices ${{v_1\in V(G_1)}}$ and ${{v_2\in V(G_2)}}$ whose neighborhoods are ${{N(v_1)=\{x_1,y_1,z_1\}}}$ and ${{N(v_2)=\{x_2,y_2,z_2\}}}$, we include in~$\mathcal{C}$ the graph
    \begin{equation*}
        (G_1-v_1)\cup (G_2-v_2) + x_1x_2 + y_1y_2 + z_1z_2.\label{thm:constr:ii}
    \end{equation*}
    \item For a graph ${{G\in\mathcal{C}}}$ with distinct vertices ${{v, w, x\in V(G)}}$, containing~${{vw\in E(G)}}$, and satisfying ${{\deg(z)=3}}$ for all ${{z\in V(G)\setminus\{x\}}}$, we include in~$\mathcal{C}$ the graph
    \begin{equation*}
        G+y-vw+vy+wy+xy,
    \end{equation*}
    where~${{y\notin V(G)}}$ is a new vertex to be added to~$G$.\label{thm:constr:iii}
\end{enumerate}
\end{theorem}

The operations~\eqref{thm:constr:ii} and~\eqref{thm:constr:iii} are illustrated in Figure~\ref{fig:operations}. We refer to~\eqref{thm:constr:ii} as a \emph{bridge operation} and to~\eqref{thm:constr:iii} as a \emph{spoke operation}. More precisely, if~${{\deg(x)=3}}$ in~\eqref{thm:constr:iii}, we call it a \emph{primary spoke operation} and if~${{\deg(x)>3}}$, we call it a \emph{secondary spoke operation}. Note that the class of $3$-regular $3$-connected graphs is contained in the class of uniformly $3$-connected graphs. In turn, the class of uniformly $3$-connected graphs is contained in the class of $3$-connected graphs. So Theorem~\ref{thm:constr} is in a sense complementary to the classical constructions by Tutte~\cite{tutte1961three,tutte1966connectivity} for $3$-regular $3$-connected and~$3$-connected graphs. A natural question to ask when learning about a class of graphs is what degrees one might see. In extremal graph theory, this led to extensive research on the minimum number of vertices of minimum degree. Formally, for a graph~$G$ one asks for the parameter
\begin{equation*}
    \nu(G)\coloneqq\big|\big\{v\in V(G):\deg(v)=\,\min_{\mathclap{v\in V(G)}}\;\deg(v)\big\}\big|.
\end{equation*}
A cornerstone on which many related investigations build on is the result by Halin~\cite{halin1969theorem}, who proved that a minimally \mbox{$k$-connected} graph contains a vertex of degree~$k$. A series of results on that topic is concluded by Mader~\cite{mader1979struktur}, who gave the tight bound~${{\nu(G)\geq \lceil((k-1)@n+2@k)/(2@k-1)\rceil}}$ for a minimally \mbox{$k$-connected} graph~$G$ on~$n$ vertices. This result does also hold for uniformly $3$-connected graphs, as those are minimally $k$-connected. See Beineke, Oellermann, and Pippert~\cite{beineke2002average} for a proof of that result. But as minimally $k$-connected graphs do not have to be uniformly $k$-connected, there can be stronger bounds on~$\nu(G)$ and indeed there is the following result.
\begin{theorem}\label{thm:bound}
A uniformly $3$-connected graph~$G$ on $n$ vertices satisfies
\begin{equation*}
    \nu(G) \ge \lceil(2n+2)/3\rceil.
\end{equation*}
\end{theorem}
This result is proven in~\cite{hofmann2022uniformly}. We call a uniformly $3$-connected graph \emph{extremal} if it attains the bound from Theorem~\ref{thm:bound}. The results of Section~\ref{sec:main} shall help us to learn more about that class. There we show in detail how the number of vertices of a uniformly $3$-connected graph depends on the operations involved in constructing it. Furthermore, we show that the bridge operation preserves in a sense crossing numbers and under certain conditions treewidths larger than two. We denote the crossing number of a graph~$G$ by~$\cro(G)$ and its treewidth by~$\tw(G)$. Section~\ref{sec:appl} is intended to demonstrate how our results can be used, for example, to find out when extremal uniformly $3$-connected graphs are planar.

\section{Main results}\label{sec:main}

In what follows, we build on one of the characterizations by Tutte~\cite[Chapter~12]{tutte1966connectivity}, which says that all~$3$-regular $3$-connected graphs can be obtained from a complete graph on four vertices by a sequence of \emph{edge joins}. Formally, for a graph~$G$ and two edges~${{st,vw\in E(G)}}$ \emph{joining} them means to build the graph
\begin{equation*}
    G+x+y-st-vw+sx+xt+vy+yw+xy
\end{equation*}
where~${{x,y\notin V(G)}}$ are new vertices to be added to~$G$. This construction is illustrated in Figure~\ref{fig:edgejoin}. Note also that~$st$ and~$vw$ are two distinct edges, but they may share one endvertex.
\begin{figure}[t]
    \centering
    \includegraphics{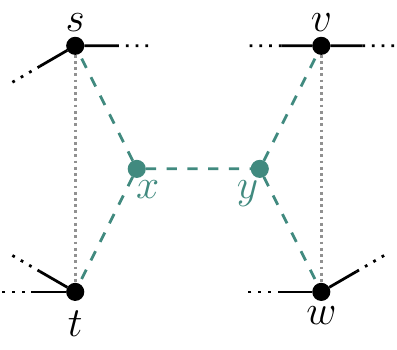}
    \caption{Joining two edges of a $3$-regular graph}
    \label{fig:edgejoin}
\end{figure}
\begin{theorem}\label{thm:main}
A uniformly $3$-connected graph $G$ on $n$ vertices satisfies
\begin{equation*}
    n = 4 + 2@j + 2@t + p + s
\end{equation*}
if~$G$ is constructed from complete graphs on four vertices by a sequence of $j$ bridge operations,~$t$ edge joins, $p$ primary spoke operations and $s$ secondary spoke operations.\vspace{0.226ex}
\end{theorem}

\hspace{-3.5mm}\textbf{Proof:}~The smallest uniformly $3$-connected graph is the complete graph on four vertices, for which ${{j=t=p=s=0}}$ and our claim holds. Now suppose we are given a graph~$G$ on~$n$ vertices and our statement is true for all graphs on less than~$n$ vertices.

First, take the case where an edge join is the final operation in the sequence of operations to build~$G$. Then~$G$ arises from a graph~$G'$ with~${{n=|V(G)|=|V(G')|+2}}$, as an edge join adds two vertices. Denoting the number of edge joins to build~$G'$ by~$t'$, we have ${{t=t'+1}}$. By induction, we obtain
\begin{align*}
    n &= |V(G)|=|V(G')|+2 \\
      &= 4 + 2@j + 2@t'+2 + p + s \\
      &= 4 + 2@j + 2@t + p + s.
\end{align*}
Primary or secondary spoke operations add one vertex, as is illustrated in Figure~\ref{fig:operations}. If such an operation is the final operation to build~$G$, we can argue as in the previous case. It remains the case where a bridge operation is the final operation to build~$G$. Then~$G$ arises from two graphs~$G_1$ and~$G_2$. In view of Figure~\ref{fig:operations}, we have~${{n=|V(G)|=|V(G_1)| + |V(G_2)| - 2}}$, as well as~${{j=j_1+j_2+1}}$, ${{t=t_1+t_2}}$, ${{p=p_1+p_2}}$, and~${{s=s_1+s_2}}$, where~$j_i,t_i,p_i,s_i$ are the respective numbers of bridge operations, edge joins, primary and secondary spoke operations used when constructing~$G_i$, where ${{i\in\{1,2\}}}$. By induction, we obtain
\begin{align*}
    n &= |V(G)|=|V(G_1)| + |V(G_2)| - 2 \\
      &= 4 + 2@j_1 + 2@t_1 + p_1 + s_1 + 4 + 2@j_2 + 2@t_2 + p_2 + s_2 - 2 \\
      \phantom{\square}\hspace*{26.746mm}\phantom{n}&= 4 + 2@(t_1+t_2) + 2@(j_1 + j_2 + 1) + (p_1+p_2) + (s_1+s_2)\hspace*{26.746mm}\phantom{\square} \\
      &= 4 + 2@t + 2@j + p + s.\negphantom{4 + 2@j + 2@t + p + s.}\phantom{4 + 2@(j_1 + j_2 + 1) + 2@(t_1+t_2) + (p_1+p_2) + (s_1+s_2)}\hspace*{26.746mm}\square
\end{align*}

This allows us to reprove Theorem~\ref{thm:bound} as well as to obtain some additional conditions on the numbers of operations involved.

\begin{proof}[of Theorem~\ref{thm:bound}] For a uniformly $3$-connected graph~$G$ on~$n$ vertices, Theorem~\ref{thm:main} tells us that
\begin{equation}
    n = 4 + 2@j + 2@t + p + s.\label{eqn:1}
\end{equation}
Let us recall that a primary spoke operation, by definition, can only be applied to $3$-regular graphs, and it raises one of the respective degrees to four. A graph whose construction involves~$j$ bridge operations is formed by recursively combining ${{j+1}}$ input graphs. For each input graph, one is allowed to use at most one primary spoke operation. In other words,
\begin{equation}
    j+1 \geq p~\Rightarrow~2@j \geq 2@p-2.\label{eqn:2}
\end{equation}
Combining Equations~\eqref{eqn:1} and~\eqref{eqn:2}, we obtain
\begin{equation}
    n \geq 2 + 2@t + 3@p +s\geq 2 + 3@p~\Rightarrow~p \leq \lfloor(n-2)/3\rfloor.\label{eqn:3}
\end{equation}
The primary spoke operation is the only operation that reduces the number of vertices of minimum degree. It does so by exactly one. Consequently,
\begin{equation}
    \nu(G) \geq n-p\geq \lceil(2@n+2)/3\rceil,\label{eqn:4}
\end{equation}
which was to be shown.
\end{proof}

Another property we shall verify in this section is that the bridge operation preserves the crossing numbers of the input graphs. In our proof, we build on the following basic fact about graph embeddings, presented by West~\cite[Chapter~6]{west2001introduction}.
\begin{lemma}\label{lem:outerface}
If~$E$ is the edge set of a face of some planar embedding of a graph~$G$, then there is an embedding of~$G$ such that~$E$ is the edge set of the outer face.
\end{lemma}
\begin{theorem}\label{thm:crossings}
If $G$ is the result of applying the bridge operation on graphs~$G_1$ and~$G_2$, then
\begin{equation*}
    \cro(G)\leq \cro(G_1) + \cro(G_2).\vspace{-1ex}
\end{equation*}
\end{theorem}
\begin{proof}
We are given two graphs~${{G_1, G_2}}$ with vertices ${{v_1\in V(G_1)}}$ and ${{v_2\in V(G_2)}}$ whose neighborhoods are ${{N(v_1)=\{x_1,y_1,z_1\}}}$ and ${{N(v_2)=\{x_2,y_2,z_2\}}}$ and a graph
    \begin{equation*}
        G\coloneqq(G_1-v_1)\cup (G_2-v_2) + x_1x_2 + y_1y_2 + z_1z_2.
    \end{equation*}
At first, let us consider some drawing of~$G_1$ in the plane, possibly with crossings. We obtain a \emph{planarization}~$P$ of this drawing by replacing each occurring crossing by a new vertex. In this process, we may have to subdivide some of the edges in~$\{x_1v_1,y_1v_1,z_1v_1,x_2v_2,y_2v_2,z_2v_2\}$. The vertex on the former edge~$x_1v_1$ excluding~$v_1$ but including~$x_1$ that is closest to $v_1$ shall be denoted by~$x_1'$. Analogously, we define $y_1',z_1',x_2',y_2',z_2'$. Since~${{\deg(v_1)=3}}$, we know that two of the three edges $x_1'v_1,y_1'v_1,z_1'v_1$, say $x_1'v_1$ and $y_1'v_1$, are both contained in the edge set of some face of~$P$. Lemma~\ref{lem:outerface} tells us that there is an embedding of~$P$ such that~$\{x_1'v_1,y_1'v_1\}$ is contained in the edge set of the outer face. Replacing the vertices we introduced when planarizing~$G$ back to crossings, we obtain a drawing of~$G_1$ where parts of both edges~$x_1v_1$ and $y_1v_1$ are incident to the outer face. Even more, since we can reflect the embedding of~$G_1$ across a line through~$v_1$, it is possible to choose the orientation of~$\{x_1v_1,y_1v_1\}$. Likewise, we can take a drawing of~$G_2$ where parts of~$x_2v_2$ and~$y_2v_2$ are incident to the outer face. In other words, our situation is essentially as in Figure~\ref{fig:crossings}.

\begin{figure}[t]
    \centering
    \includegraphics{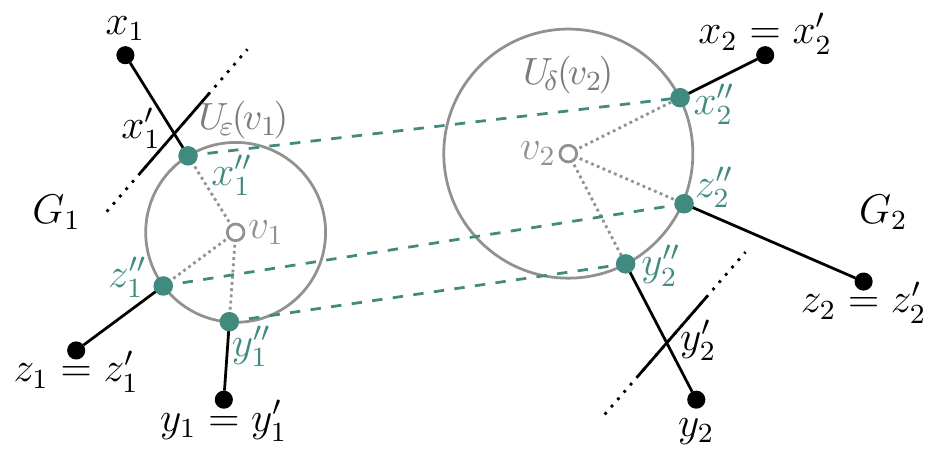}
    \caption{The bridge operation acting on graphs embedded in the plane}
    \label{fig:crossings}
\end{figure}
Since we embedded finite graphs in the plane, we can find radii~${{\varepsilon,\delta>0}}$ such that the discs~$U_\varepsilon(v_1)$ and~$U_\delta(v_2)$ do not contain $x_1',y_1',z_1',x_2',y_2'$, or $z_2'$. We denote the intersection of the edge~$x_1v_1$ with the disc~$U_\varepsilon(v_1)$ by~$x_1''$ and the intersection of the edge~$x_2v_1$ with the disc~$U_\delta(v_2)$ by~$x_2''$. This provides us with a polygonal arc, leading from $x_1$ to $x_1''$ to $x_2''$ to $x_2$. There are analogous polygonal arcs linking $y_1$ with~$y_2$ and $z_1$ with~$z_2$. Those polygonal arcs can be drawn without intersections when choosing the orientation of the embeddings of~$G_1$ or~$G_2$ as in Figure~\ref{fig:crossings}. This tells us that we can build~$G$ out of~$G_1$ and~$G_2$ by the bridge operation without adding any additional crossings. So~${{\cro(G)\leq\cro(G_1) + \cro(G_2)}}$.
\end{proof}

Finally, we will ask how the bridge operation affects the treewidths of the input graphs. So let us recall the following terms.

\begin{definition}\label{def:trd}
A \emph{tree decomposition} of a graph~$G$ is a pair ${{(\{X_i:i\in I\},T=(I,F))}}$ where $T$ is a tree and each \emph{node} $i\in I$ has a \emph{bag} ${{X_i\subseteq V(G)}}$ such that the following properties hold.
\begin{enumerate}
    \item Each vertex of~$V$ belongs to some bag, or~${{\cup_{i\in I}X_i = V}}$.\label{def:trd:i}
    \item For all ${{vw\in E(G)}}$ there exists an~$i\in I$ such that ${{v,w\in X_i}}$.\label{def:trd:ii}
    \item For all~${{v\in V}}$ the set of nodes ${{\{i\in I:v\in X_i\}}}$ induces a subtree of~$T$.\label{def:trd:iii}
\end{enumerate}
The \emph{width} of a tree decomposition ${{(\{X_i:i\in I\},T=(I,F))}}$ is $\max_{i\in I} |X_i|-1$ and the \emph{treewidth} of a graph~$G$ is the minimum width of all tree decompositions of~$G$. We shall denote the latter by~$\tw(G)$.
\end{definition}
Before we focus on how the treewidth behaves under the  bridge operation, let us recall the following facts, whose proofs can be found in Bodlaender~\cite{bodlaender1998partial}.
\begin{lemma}\label{lem:minors}
If~$H$ is a minor of~$G$, then~${{\tw(H)\leq \tw(G)}}$.
\end{lemma}
\begin{lemma}\label{lem:cliques}
If~${{(\{X_i:i\in I\},T=(I,F))}}$ is a tree decomposition of a graph~$G$ and~${{W\subseteq V(G)}}$ a clique in~$G$, then there is a node~${{i\in I}}$ such that~${{W\subseteq X_i}}$.
\end{lemma}

Furthermore, given a graph~$G$, we call a vertex~${{v\in V(G)}}$ with~${{\deg(v)=3}}$ \emph{safe} if~$G$ admits a minimum width tree decomposition having a bag that contains~$v$ and two of its neighbors. In view of Lemma~\ref{lem:cliques}, a vertex of degree three is safe if and only if it has two neighbors that are adjacent or that can be joined by an edge without increasing the treewidth of~$G$. Furthermore, we call a vertex of degree three \emph{unsafe} if it is not safe. By definition, an unsafe vertex has an independent neighborhood, as is the case for the vertex~$v$ in Figure~\ref{fig:cliqueAtunsafe}. Suppose~$v$ is an unsafe vertex of the indicated graph~$G$, with neighborhood~${{N(v)=\{x,y,z\}}}$. Then adding a clique on four vertices by the bridge operation results in a graph that has~${{G+xy}}$ as minor, which can be seen by contracting the vertices shaded in gray. So, in general, the bridge operation can increase the treewidth, but only if we combine graphs at unsafe vertices, as we will show next.
\begin{figure}[t]
    \centering
    \includegraphics{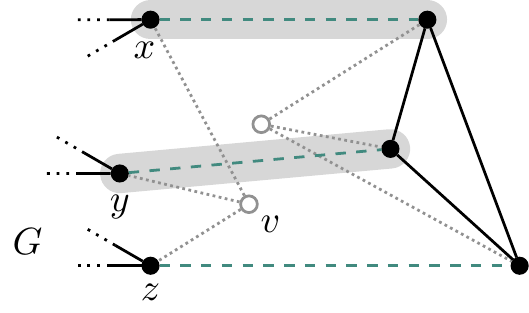}
    \caption{Adding a clique at an unsafe vertex}
    \label{fig:cliqueAtunsafe}
\end{figure}%
\begin{theorem}\label{thm:treewidth}
Consider two graphs~$G_1$ and~$G_2$ with vertices ${{v_1\in V(G_1)}}$ and ${{v_2\in V(G_2)}}$ whose neighborhoods are ${{N(v_1)=\{x_1,y_1,z_1\}}}$ and ${{N(v_2)=\{x_2,y_2,z_2\}}}$ and a graph
    \begin{equation*}
        G\coloneqq(G_1-v_1)\cup (G_2-v_2) + x_1x_2 + y_1y_2 + z_1z_2.
    \end{equation*}
If $v_1$ and~$v_2$ are safe and~${{\max\{\tw(G_1),\tw(G_2)\}\geq 3}}$, then
\begin{equation*}
    \tw(G)=\max\{\tw(G_1),\tw(G_2)\}.
\end{equation*}
\end{theorem}
\begin{proof}
Let~${{(\{X_i:i\in I_1\},T_1=(I_1,F_1))}}\eqqcolon (\mathcal{X},T_1)$ and ${{(\{Y_j:j\in I_2\},T_2=(I_2,F_2))\eqqcolon (\mathcal{Y},T_2)}}$ be minimum width tree decompositions of~$G_1$ and~$G_2$, respectively, having a bag~${{X_s\in\mathcal{X}}}$ containing~$v_1$ and two of its neighbors, say~$x_1$ and $y_1$, and a bag~${{Y_t\in\mathcal{Y}}}$ containing $v_2$ and two of its neighbors. We can assume the existence of such bags because~$v_1$ and~$v_2$ are safe. To verify ${{\tw(G)\leq\max\{\tw(G_1),\tw(G_2)\}}}$, our goal is to define a tree decomposition of width at most~${{\max\{\tw(G_1),\tw(G_2)\}}}$ for $G$. Whereas we have denoted the neighbors of~$v_1$ in bag~$X_s$ by~$x_1$ and~$y_1$ without loss of generality, there are two cases to consider with respect to how the vertices of~$X_s$ and~$Y_t$ are joined by edges in~$G$. Setting~${{F\coloneqq\{x_1x_2,y_1y_2,z_1z_2\}\cap E(G[X_s\cup Y_t])}}$, either~${{|F|=1}}$ or~${{|F|\geq 2}}$.

Let us begin with the case~${{|F|=1}}$, denoting the neighbors of~$v_2$ in~$G_2$ that are contained in~$Y_t$ by~$x_2$ and~$z_2$. Since in~$G$ the vertices~$v_1$ and~$v_2$ do not exist, we may safely replace them. Formally, for each~${{i\in I_1}}$ where~${{v_1\in X_i}}$ set~${{X_i'\coloneqq X_i\setminus\{v_1\}\cup\{z_2\}}}$ and for each~${{i\in I_1}}$ where~${{v_1\notin X_i}}$ set~${{X_i'\coloneqq X_i}}$. Furthermore, for each~${{j\in I_2}}$ where~${{v_2\in Y_j}}$ set~${{Y_j'\coloneqq Y_j\setminus\{v_2\}\cup \{y_1\}}}$ and for each~${{j\in I_2}}$ where~${{v_2\notin Y_j}}$ set~${{Y_j'\coloneqq Y_j}}$. Note that we have not increased the cardinalities of the bags. Now take a new node~${{v\notin I_1\cup I_2}}$ to define the tree ${{T\coloneqq T_1 \cup T_2 + v + sv + vt}}$ as well as the bag~${{X_v\coloneqq\{x_1,x_2,y_1,z_2\}}}$. Because~${{|X_v|=4}}$ and our assumption that~${{\max\{\tw(G_1),\tw(G_2)\}\geq 3}}$, we observe
\begin{equation*}
    \max\big\{\max\limits_{i\in I_1}|X_i|,\max\limits_{j\in I_2}|Y_j|\big\} = \max\big\{\max\limits_{i\in I_1}|X_i'|,\max\limits_{j\in I_2}|Y_j'|,|X_v|\big\}.
\end{equation*}
\begin{figure}[t]
    \centering
    \includegraphics{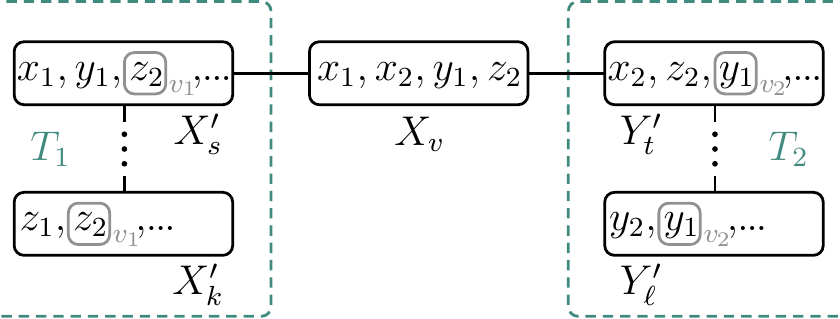}
    \caption{Combining two tree decompositions at bags of safe vertices when~${{|F|=1}}$}
    \label{fig:treewidth_bridge_1}
\end{figure}%
It remains to be checked that~${{D\coloneqq(\{X_i':i\in I_1\}\cup\{Y_j':j\in I_2\}\cup \{X_v\},T)}}$ is a tree decomposition of~$G$. When building the bags of~$D$, the only vertices we removed were~$v_1$ and~$v_2$, which are not present in~$G$. So~$D$ satisfies Condition~\ref{def:trd:i} of Definition~\ref{def:trd}. By the same reason, for each edge in~${{E(G_1)\cup E(G_2)}}$ we find a bag in~$D$ containing its endvertices. Furthermore, by Condition~\ref{def:trd:ii} of Definition~\ref{def:trd}, there must be some~${{k\in I_1}}$ such that ${{v_1,z_1\in X_k}}$, which implies that ${{z_1,z_2\in X_k'}}$. Likewise, there is an~${{\ell\in I_2}}$ such that~${{y_1,y_2\in Y_\ell'}}$. Since the edge~$x_1x_2$ is covered by the bag~$X_v$, this verifies Condition~\ref{def:trd:ii} of Definition~\ref{def:trd}. We have to check Condition~\ref{def:trd:iii} of Definition~\ref{def:trd} essentially for the vertices in~$X_v$. This is because~$T$ by construction is a tree having~$T_1$ and~$T_2$ as subtrees, the only vertices we removed when building~$D$ were~$v_1$ and~$v_2$, and the only vertices we included in some bag were those of~$X_v$. Figure~\ref{fig:treewidth_bridge_1} illustrates the construction of the tree decomposition. Herein, we placed~$z_2$ in every bag that contained~$v_1$, indicated by~$z_2$ in a gray box with subscript~$v_1$. Therefore, we observe that ${{\{i\in I_1:z_2\in X_i'\}}}$ induces a subtree of~$T_1$. Since ${{\{j\in I_2:z_2\in Y_j'\}=\{j\in I_2:z_2\in Y_j\}}}$ induces a subtree of~$T_2$ and~${{z_2\in X_v}}$, we find that the nodes whose bags contain~$z_2$ induce a subtree of~$T$. By investigating Figure~\ref{fig:treewidth_bridge_1}, we can argue similarly for the remaining vertices of~$X_v$.

For the case~${{|F|=2}}$, denote the neighbors of~$v_2$ in~$G_2$ that are contained in~$Y_t$ by~$x_2$ and~$y_2$. For each~${{i\in I_1}}$ where~${{v_1\in X_i}}$ set~${{X_i'\coloneqq X_i\setminus\{v_1\}\cup \{z_1\}}}$ and for each~${{i\in I_1}}$ where~${{v_1\notin X_i}}$ set~${{X_i'\coloneqq X_i}}$. Likewise, for each~${{j\in I_2}}$ where~${{v_2\in Y_j}}$ set~${{Y_j'\coloneqq Y_j\setminus\{v_2\}\cup \{z_1\}}}$ and for each~${{j\in I_2}}$ where~${{v_2\notin Y_j}}$ set~${{Y_j'\coloneqq Y_j}}$. For two new nodes~${{v,w\notin I_1\cup I_2}}$ define the tree ${{T\coloneqq T_1 \cup T_2 + v + w + sv + vw + wt}}$ as well as the bags~${{X_v\coloneqq\{x_1,y_1,y_2,z_1\}}}$ and ${{X_w\coloneqq\{x_1,x_2,y_2,z_1\}}}$. This defines a tree decomposition of width at most~${{\max\{\tw(G_1),\tw(G_2)\}}}$, which can be checked by investigating Figure~\ref{fig:treewidth_bridge_2}, analogous to the previous case.
\begin{figure}[t]
    \centering
    \includegraphics{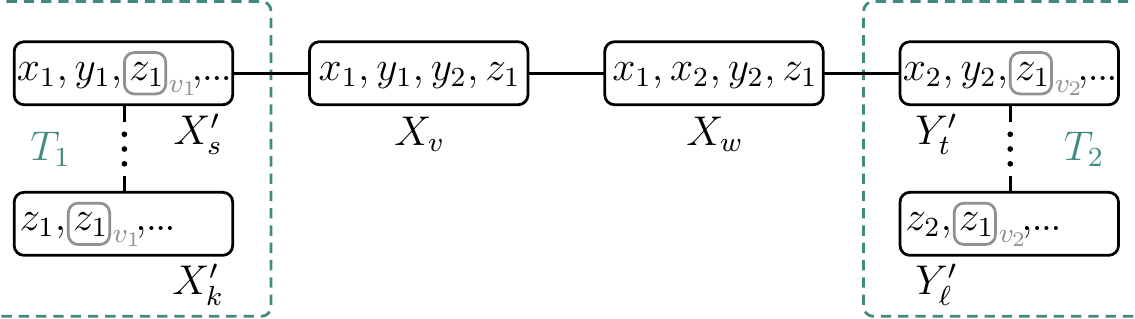}
    \caption{Combining two tree decompositions at bags of safe vertices when~${{|F|=2}}$}
    \label{fig:treewidth_bridge_2}
\end{figure}%

For the other inequality, note that both~$G_1$ and~$G_2$ are minors of~$G$. For example, contracting all vertices in~$G$ that stem from~$G_2$ to a single vertex yields~$G_1$. This implies~${{\tw(G)\geq\max\{\tw(G_1),\tw(G_2)\}}}$ by Lemma~\ref{lem:minors}, which concludes our proof.
\end{proof}

Quite a few difficult combinatorial problems on graphs can be solved in polynomial, or even linear, time by dynamic programming approaches if the input graph has bounded treewidth, about which Bodlaender and Koster~\cite{bodlaender2008combinatorial} give an overview. This makes statements such as that of Theorem~\ref{thm:treewidth} useful. In what follows, however, we will encounter situations where we cannot assume the vertices involved in our bridge construction to be safe. Nevertheless, there are some tools that will help us to show that extremal uniformly \mbox{$3$-connected} graphs have bounded treewidth. To this end, let us recall the notion of a \emph{line graph}~$L(G)$ of a graph~$G$. This is the graph on vertex set~$E(G)$ whose vertices are adjacent exactly when they are incident in~$G$.

\begin{lemma}\label{lem:twLG}
For every graph $G$, we have
\begin{equation*}
    \tw(G)\leq 2\tw(L(G))+1.
\end{equation*}
\end{lemma}
This bound and related results are presented by Harvey and Wood~\cite{harvey2018treewidth}. Furthermore, Bodlaender, Van Leeuwen, Tan, and Thilikos~\cite{bodlaender1997interval} give the following relation.
\begin{lemma}\label{lem:cliquesum}
Let~$G_1$ and~$G_2$ be two graphs containing cliques~${{S\subseteq V(G_1)}}$ and~${{T\subseteq V(G_2)}}$ with~${{|S|=|T|}}$ and let~$G$ be a \emph{clique-sum} of~$G_1$ and~$G_2$, meaning a graph obtained by taking the disjoint union of~$G_1$ and~$G_2$ and identifying~$S$ and~$T$. Then
\begin{equation*}
    \tw(G) = \max\{\tw(G_1),\tw(G_2)\}.
\end{equation*}
\end{lemma}
\begin{figure}[t]
    \centering
    \includegraphics{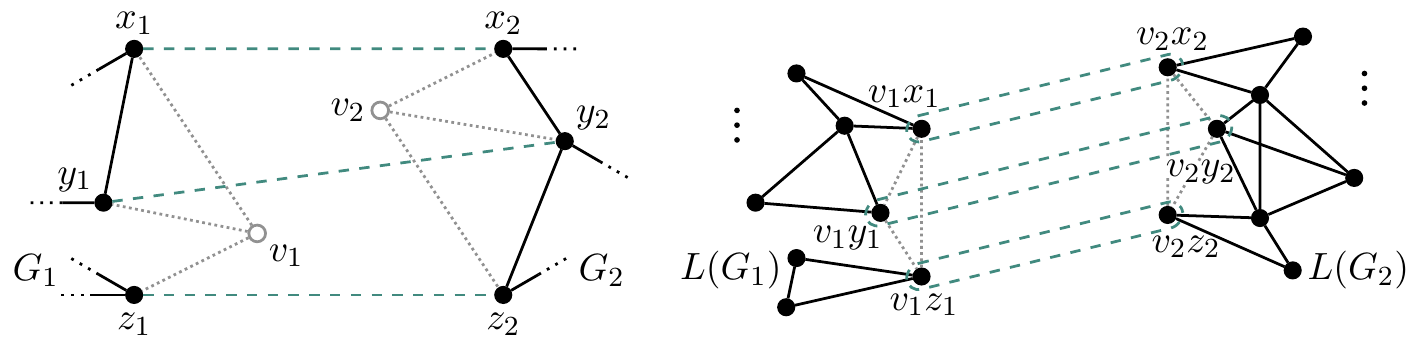}
    \caption{The bridge operation and its effect on the corresponding line graphs}
    \label{fig:bridge_linegraph}
\end{figure}%
\begin{lemma}\label{lem:bridgeclique-sum}
Consider two graphs~$G_1$ and~$G_2$ with vertices ${{v_1\in V(G_1)}}$ and ${{v_2\in V(G_2)}}$ whose neighborhoods are ${{N(v_1)=\{x_1,y_1,z_1\}}}$ and ${{N(v_2)=\{x_2,y_2,z_2\}}}$ and a graph
\begin{equation*}
    G\coloneqq(G_1-v_1)\cup (G_2-v_2) + x_1x_2 + y_1y_2 + z_1z_2.
\end{equation*}
Furthermore, let~$H$ be a clique-sum of $L(G_1)$ and $L(G_2)$ formed by identifying ${{E(\{v_1\},V(G_1)\setminus\{v_1\})}}$ and ${{E(\{v_2\},V(G_2)\setminus\{v_2\})}}$. Then $L(G)$ is a (proper) subgraph of~$H$.
\end{lemma}
\begin{proof}
When forming~$G$ by the bridge operation, adding the edges~$x_1x_2$,~$y_1y_2$, and~$z_1z_2$ corresponds to identifying $v_1x_1$ with $v_2x_2$, $v_1y_1$ with $v_2y_2$, and $v_1z_1$ with $v_2z_2$ in the respective line graphs, as is indicated by dashed green lines in Figure~\ref{fig:bridge_linegraph}. Deleting~$v_1$ and~$v_2$ when forming~$G$ by the bridge operation removes the cliques, indicated by dotted gray lines in Figure~\ref{fig:bridge_linegraph}, at which the clique-sum of $L(G_1)$ and $L(G_2)$ is formed. This is why $L(G)$ is a \emph{proper} subgraph of~$H$.
\end{proof}
\begin{theorem}\label{thm:twbridgeclass}
Let~$\mathcal{C}$ be a class of graphs which arises by successively taking the bridge operation to join graphs from a base class whose line graphs have treewidth bounded by $w$. Then for every graph~$G\in\mathcal{C}$, we have
\begin{equation*}
    \tw(G)\leq 2w+1.
\end{equation*}
\end{theorem}
\begin{proof}
This follows directly from Lemmas~\ref{lem:twLG},~\ref{lem:cliquesum}, and~\ref{lem:bridgeclique-sum}.
\end{proof}

\section{Applications}\label{sec:appl}

Let us proceed with an example that illustrates how to use Equations~\eqref{eqn:1} to~\eqref{eqn:4}, which we obtained in the course of our proof of Theorem~\ref{thm:bound}, to get a precise picture of extremal uniformly $3$-connected graphs.

\begin{example} \textnormal{Let us ask for the graphs on~${{n=10}}$ vertices with minimum number of vertices of minimum degree. Condition~\eqref{eqn:4} tells us that the extremal graphs are those where $p$ is maximal. In view of Condition~\eqref{eqn:3}, we choose~${{p = 2}}$. Condition~\eqref{eqn:1} then reads~${{4 = 2@t + 2@j + s}}$ and by Condition~\eqref{eqn:2}, we obtain~${{j \geq 1}}$. This leaves us exactly with the settings where~$p=2$ and
\begin{equation*}
    t=1, j=1, s=0\quad\text{or}\quad t=0, j=2, s=0\quad\text{or}\quad t=0, j=1, s=2.
\end{equation*}
A graph for the setting ${{t=1, j=1, p=2, s=0}}$ is illustrated in Figure~\ref{fig:extremalnonplanar}.}
\begin{figure}[b]
    \centering
    \includegraphics{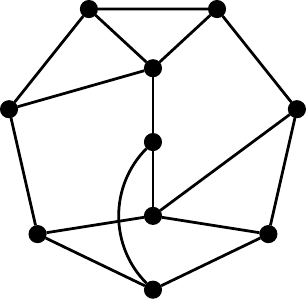}
    \caption{An extremal uniformly $3$-connected graph on ten vertices}
    \label{fig:extremalnonplanar}
\end{figure}
\end{example}
In what follows, we shall generalize the findings from this example, and so identify the conditions under which extremal uniformly $3$-connected graphs are planar.
\begin{theorem}\label{thm:numbers}
Given an extremal uniformly $3$-connected graph on~${{n=3@k+\ell\geq 5}}$ vertices, for some ${{k\in\mathbb{N}\setminus\{1\}}}$ and~${{\ell\in\{-1,0,1\}}}$, let~$j,t,p$, and~$s$ be the respective numbers of bridge operations, edge joins, primary and secondary spoke operations involved in constructing~$G$.\newpage
\begin{enumerate}
    \item Then~$p=k-1$.\label{thm:numbers:i}
    \item If $\ell=-1$, then $j=k-2$,~$t=s=0$.\label{thm:numbers:ii}
    \item If $\ell=\textcolor{white}{-}0$, then ${{j=k-2}}$,~${{t=0}}$,~${{s=1}}$.\label{thm:numbers:iii}
    \item If $\ell=\textcolor{white}{-}1$, then ${{j=k-1}}$,~${{t=s=0}}$ or ${{j=k-2}}$,~${{t=1}}$,~${{s=0}}$\\
    \phantom{If $\ell=\textcolor{white}{-}1$, then ${{j=k-1}}$,~${{t=s=0}}$ }or ${{j=k-2}}$,~${{t=0}}$,~${{s=2}}$.\label{thm:numbers:iv}
\end{enumerate}
\end{theorem}
\begin{proof}
In view of Conditions~\eqref{eqn:3} and~\eqref{eqn:4}, building an extremal graph involves
\begin{equation*}
    p=\lfloor(n-2)/3\rfloor=\lfloor(3@k+\ell-2)/3\rfloor=k+\lfloor(\ell-2)/3\rfloor=k-1
\end{equation*} primary spoke operations. Thus Statement~\ref{thm:numbers:i} holds. Condition~\eqref{eqn:2} requires that ${{j\geq p-1= k-2}}$ and so Condition~\eqref{eqn:1} tells us that
\begin{align*}
                n &= 4 + 2@t + 2@j + p + s \\
    \Rightarrow 3@k+\ell &\geq 4 + 2@t + 2@(k-2) + k-1 + s \\
    \Rightarrow \phantom{3@k+\ell}\negphantom{1+\ell}1+\ell &\geq 2@t + s.
\end{align*}
For~${{\ell=-1}}$, we obtain~${{j=k-2}}$,~${{t=s=0}}$, which is Statement~\ref{thm:numbers:ii}. For~${{\ell=0}}$, we obtain~${{j=k-2}}$,~${{t=0}}$, ${{s=1}}$, which is Statement~\ref{thm:numbers:iii}. For~${{\ell=1}}$ and~${{j=k-2}}$, we obtain~${{t=0}}$ and~${{s=2}}$ or~${{t=1}}$ and~${{s=0}}$, which are the last two alternatives in Statement~\ref{thm:numbers:iv}. If~${{\ell=1}}$ and~${{j=k-1}}$, then Condition~\eqref{eqn:1} implies~$t=s=0$, which is the remaining alternative in Statement~\ref{thm:numbers:iv}. Finally, note that $j$ cannot be larger than ${{k-1}}$, since otherwise the right hand side of Equation~\eqref{eqn:1} exceeds the left hand side.
\end{proof}

Let us see what we now know about small extremal uniformly $3$-connected graphs.
\begin{observation}\label{obs}
\textnormal{The \emph{wheel} graph on~${{n\geq 4}}$ vertices is the graph resulting from a cycle on~${{n-1}}$ vertices by adding a new vertex which is adjacent to all other vertices. We denote such a graph by~$W_n$. The graph~$W_4$ is complete and all its vertices have degree three. Performing a primary spoke operation on~$W_4$ results in~$W_{5}$. Similarly, performing a secondary spoke operation on $W_n$ results in~$W_{n+1}$ for all~${{n\in\mathbb{N}}}$.}
\begin{figure}[t]
    \centering
    \includegraphics{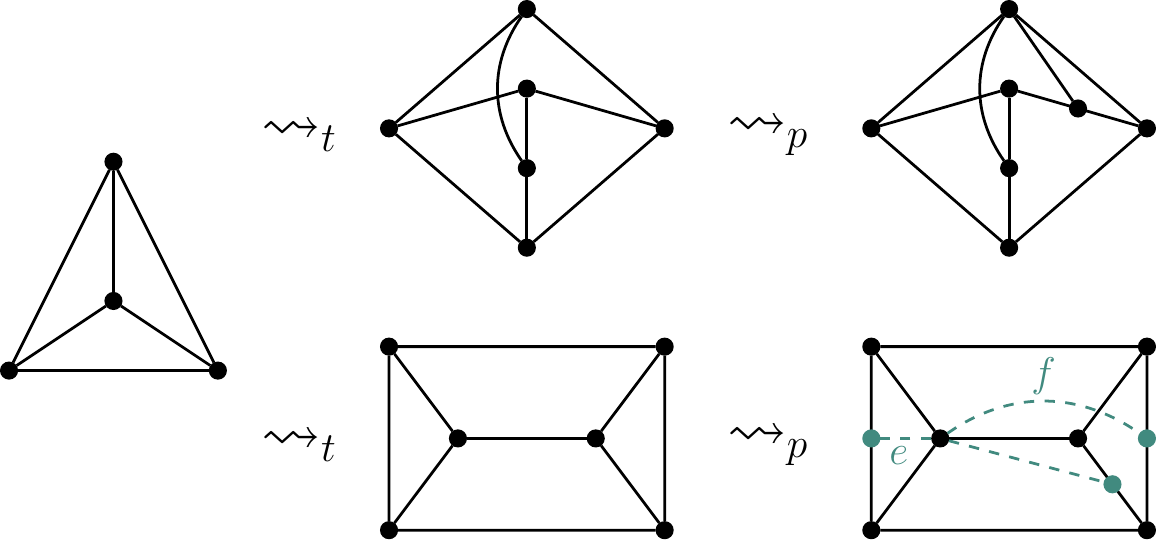}
    \caption{Small extremal uniformly $3$-connected graphs built out of a wheel graph by edge joins ($\rightsquigarrow_t$) and primary spoke operations ($\rightsquigarrow_p$) }
    \label{fig:smallextremal}
\end{figure}

\textnormal{Let us consider an extremal uniformly $3$-connected graph~$G$ in whose construction an edge join is involved. Recall that edge joins in Tutte's characterization~\cite{tutte1966connectivity}, and so in Theorem~\ref{thm:constr}, are only allowed to be applied on $3$-regular $3$-connected graphs. It is not hard to see that extremal uniformly $3$-connected graphs are nonregular for all~${{n\geq 5}}$. So when an edge join is involved in building~$G$, it can only take the graph~$W_4$ as input. This can only produce the complete bipartite graph~$K_{3,3}$ or the envelope graph, depicted in the middle of Figure~\ref{fig:smallextremal}. Out of those graphs, we can obtain the graphs on the right in Figure~\ref{fig:smallextremal} by a primary spoke operation. The dashed green edges drawn in the bottom right graph are to be understood as alternatives. They indicate the three nonisomorphic graphs that can be built out of the envelope graph by a primary spoke operation. In fact, one can check that the alternative where edge~$f$ is added to the envelope graph is isomorphic to the top right graph in Figure~\ref{fig:smallextremal}. The alternative where edge~$e$ is added to the envelope graph is isomorphic to  the graph which results from combining the wheel graphs~$W_4$ and~$W_5$ by the bridge operation. Similarly, the envelope graph can be combined out of two wheel graphs~$W_4$ by the bridge operation. So nonplanar graphs might arise even if we forbid edge joins.}
\end{observation}
With Theorem~\ref{thm:crossings}, we have the key to combine our present findings as follows.
\begin{theorem}\label{thm:crossings_extremal}
Let~$G$ be an extremal uniformly $3$-connected graph on~${{n = 3 k + \ell \geq  4}}$ vertices, for suitable~$k\in\mathbb{N}$ and~${{\ell\in\{-1,0,1\}}}$. Then ${{\cro(G)\leq 1}}$, and if ${{n = 4}}$ or ${{\ell\in\{-1,0\}}}$, then $G$ is planar.\vspace{-1ex}
\end{theorem}
\begin{proof}
The only uniformly $3$-connected graph for~${{n=4}}$ is the complete graph on four vertices. It is an extremal one and it is planar. Consider now an extremal uniformly $3$-connected graph~$G$ on~${{n=3@k+\ell\geq 5}}$ vertices, where~${{k\in\mathbb{N}\setminus\{1\}}}$. If~${{\ell\in\{-1,0\}}}$, then Items \ref{thm:numbers:i} to \ref{thm:numbers:iii} of Theorem~\ref{thm:numbers} tell us that~$G$ is built by~${{k-1}}$ primary spoke operations, one secondary spoke operation if ${{\ell=0}}$, and ${{k-2}}$ bridge operations. In other words, $G$ results from using the bridge operation recursively to combine wheels~$W_5$, and one wheel~$W_6$ if~${{\ell=0}}$. So~$G$ is planar by Theorem~\ref{thm:crossings}.

If~${{\ell=1}}$, then Items~\ref{thm:numbers:i}~and~\ref{thm:numbers:iv} of Theorem~\ref{thm:numbers} tell us that~$G$ is built by~${{k-1}}$ primary spoke operations. If~${{j=k-1}}$, then~${{t=s=0}}$. So~$G$ results from recursively using the bridge operation to combine one wheel~$W_4$ and ${{k-1}}$ wheels~$W_5$ or, in view of Observation~\ref{obs}, to combine one of the graphs in the bottom right corner of Figure~\ref{fig:smallextremal} with~${{k-2}}$ wheels~$W_5$. So~$\cro(G)\leq 1$ by Theorem~\ref{thm:crossings}.

It remains the case where~${{\ell=1}}$ and~${{j=k-2}}$. If~${{t=1}}$, then ${{s=0}}$ and $G$ results from using the bridge operation recursively to combine wheels~$W_5$ with one of the graphs on the right of Figure~\ref{fig:smallextremal}. So~${{\cro(G)\leq1}}$ by Theorem~\ref{thm:crossings}. If~${{t=0}}$, then ${{s=2}}$ and $G$ results from using the bridge operation recursively to combine wheels~$W_5$ with two~$W_6$ or one~$W_7$. So~${{\cro(G)\leq1}}$ by Theorem~\ref{thm:crossings}.
\end{proof}

\begin{figure}[t]
    \centering
    \includegraphics[valign=t]{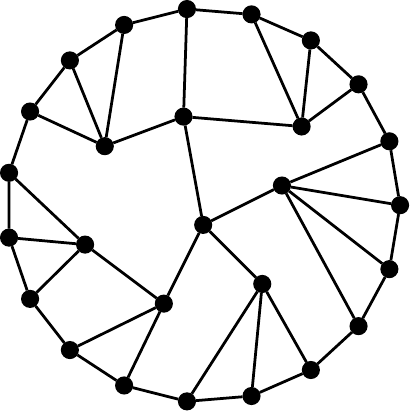}\hspace{20mm}
    \includegraphics[valign=t]{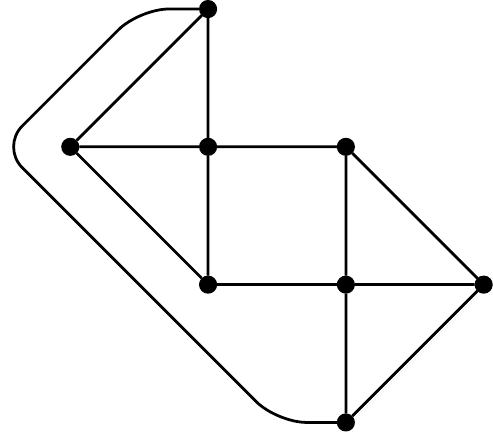}
    \caption{Extremal uniformly $3$-connected graphs}
    \label{fig:halin}
\end{figure}

Questions concerning the colorability of uniformly $3$-connected graphs are addressed by Aboulker, Brettell, Havet, Marx, and Trotignon~\cite{aboulker2017coloring}. There, it is shown that uniformly $3$-connected graphs, except the wheels on an even number of vertices, are $3$-colorable. Moreover, the authors demonstrate that such a coloring can be determined in polynomial time.

Another aspect one may notice is a certain similarity between extremal uniformly $3$-connected graphs and \emph{Halin graphs}, surveyed by Brandstadt, Le, and Spinrad~\cite{brandstadt1999graph}. Those are graphs that can be obtained by embedding a tree without vertices of degree two in the plane and connecting its leaves by a cycle without crossing any of the tree edges. By the previous proof, we can obtain those Halin graphs where the inner vertices are of degree four, with few exceptions. If~${{\ell=0}}$, we may have one vertex of degree five. If~${{\ell=1}}$, we may have two vertices of degree five or one of degree six. An example is illustrated on the left in Figure~\ref{fig:halin}. In general, Halin graphs can be seen to be uniformly $3$-connected, but not the other way around. Counterexamples are certainly nonplanar (extremal) uniformly $3$-connected graphs and even for~${{\ell=-1}}$, we find for example the graph depicted on the right in Figure~\ref{fig:halin}.
\begin{figure}[b]
    \centering
    \includegraphics{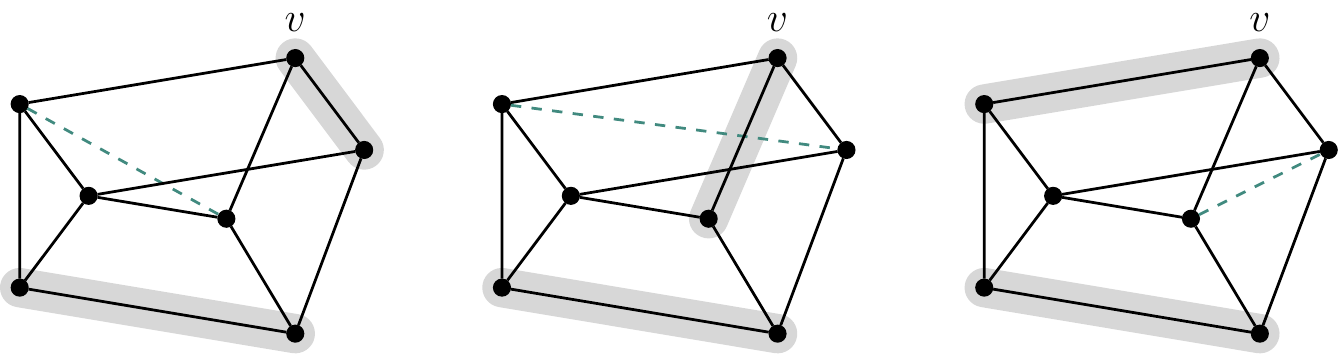}
    \caption{An extremal uniformly $3$-connected graph containing an unsafe vertex~$v$}
    \label{fig:unsavevertex}
\end{figure}%
The overlap with the class of Halin graphs motivates the question of whether extremal uniformly \mbox{$3$-connected} graphs have a similar tree-like structure. Whereas, as Bodlaender \cite{bodlaender1988planar} shows, Halin graphs have treewidth three, general uniformly \mbox{$3$-connected} graphs have unbounded treewidth. Meeks~\cite{meeks2016challenges} demonstrates this by an example illustrating that for any~${{k\in\mathbb{N}}}$ there are $3$-regular, $3$-connected graphs having a ${{k\times k}}$ grid as minor. In contrast, the small extremal uniformly $3$-connected graphs in Figure~\ref{fig:smallextremal} as well as wheel graphs have treewidth three. In view of Observation~\ref{obs} and the proof of Theorem~\ref{thm:crossings}, those are the elemental building blocks for the bridge operation when constructing extremal uniformly $3$-connected graphs. Theorem~\ref{thm:treewidth} guarantees that the bridge operation preserves the treewidth of the input graphs if no unsafe vertices arise in the course of the construction. But Figure~\ref{fig:unsavevertex} shows an extremal uniformly $3$-connected graph that has an unsafe vertex. The depicted graph is that from the bottom right corner of Figure~\ref{fig:smallextremal} where edge~$f$ is added. In Figure~\ref{fig:unsavevertex}, the neighborhood of~$v$ is independent and connecting any of its neighbors, indicated by the dashed green lines, produces a~$K_5$ minor, which can be seen by contracting the respective vertices shaded gray. As illustrated in Figure~\ref{fig:cliqueAtunsafe}, adding a wheel by the bridge operation at~$v$, produces a graph of treewidth four. Although this is the only unsafe situation we identified so far, we are not sure whether others exist, precluding us to show that the treewidth is bounded by four, using Theorem~\ref{thm:treewidth} only. However, with Theorem~\ref{thm:twbridgeclass}, we can at least verify bounded treewidth.

\begin{corollary}[of Theorem~\ref{thm:twbridgeclass}]\label{cor:tw13}
The treewidth of any extremal uniformly $3$-connected graph~$G$ is bounded by~${{\tw(G)\leq 13}}$.
\end{corollary}
\begin{figure}[t]
    \centering
    \includegraphics{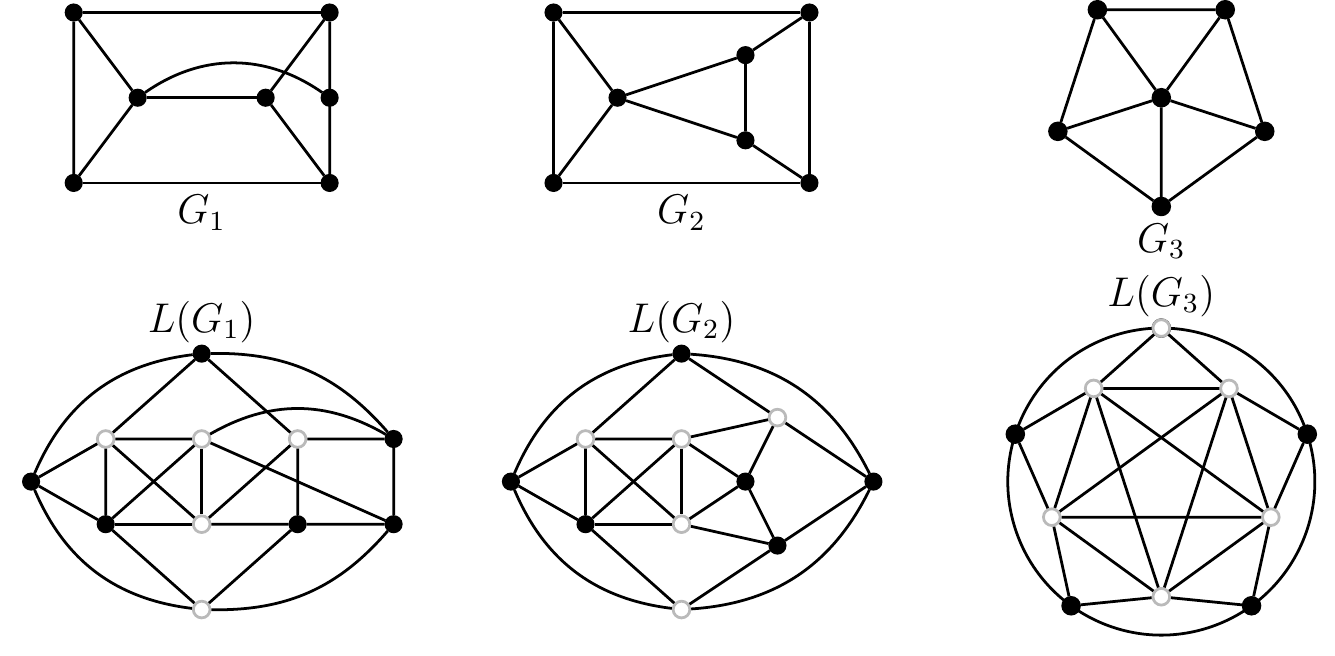}
    \caption{Small extremal uniformly $3$-connected graphs and their corresponding line graphs}
    \label{fig:small_extremal_line_graphs}
\end{figure}%
\begin{proof}
In Observation~\ref{obs} and the proof of Theorem~\ref{thm:crossings_extremal}, we argue that any extremal uniformly \mbox{$3$-connected} graph can be built by successively taking the bridge operation to join graphs from Figure~\ref{fig:smallextremal} and wheel graphs on up to six vertices. The treewidth of the line graphs of those graphs is bounded by~${{w=6}}$. This can be seen by investigating Figure~\ref{fig:small_extremal_line_graphs}. The graphs shown there satisfy~${{\tw(L(G_1))\leq 5}}$, ${{\tw(L(G_2))\leq 5}}$, and~${{\tw(L(G_3))\leq 6}}$. To certify this, let us determine respective tree decompositions. In all three graphs, removing the vertices highlighted by gray circles, and incident edges, leaves us with a tree, for which we easily find a tree decomposition of width one. Putting the vertices highlighted by gray circles in every bag, provides us with suitable tree decompositions. Also recall that the graph from the bottom right corner of Figure~\ref{fig:smallextremal} where edge~$e$ is added can be obtained by joining a wheel on five vertices with one on four vertices by the bridge operation. Clearly, all remaining line graphs of graphs from Figure~\ref{fig:smallextremal} as well as line graphs of wheels on four and five vertices are minors of one of the graphs depicted in Figure~\ref{fig:small_extremal_line_graphs}. So Theorem~\ref{thm:twbridgeclass} provides us with the asserted bound.
\end{proof}

\noindent\textbf{Open Question.} In view of Corollary~\ref{cor:tw13} and Figures~\ref{fig:unsavevertex} and~\ref{fig:cliqueAtunsafe}, we know that there is a general upper bound~${{4\leq C \leq 13}}$ such that~${{\tw(G)\leq C}}$ holds for any extremal uniformly $3$-connected graph~$G$. It is open what the smallest such bound~$C$ is. We tend to believe that~${{C=4}}$.


\nocite{*}
\bibliographystyle{plainnat}
\bibliography{bibliography}
\label{sec:biblio}

\end{document}